\documentclass[12pt,twoside]{amsart}          
\usepackage[latin1]{inputenc}
\usepackage[T1]{fontenc}
\usepackage{amsmath, amssymb, amsthm,enumitem}            
\usepackage{amstext, amsfonts, a4}

 \usepackage[colorlinks=true]{hyperref}
\hypersetup{urlcolor=blue, citecolor=red}

\usepackage{color,transparent}
\usepackage{graphics,graphicx,subfigure}

\newcommand{\N}{ \mathbb{N}}
\newcommand{\Z}{ \mathbb{Z}}

\newcommand{\R}{ \mathbb{R}}

\newcommand{\Top}{\mathrm{Top}}

\theoremstyle{plain}
    \newtheorem{teo}{Theorem}[section]
    \newtheorem{lem}[teo]{Lemma}
    
	\newtheorem{prop}[teo]{Proposition}
    \newtheorem{de}[teo]{Definition}
     \newtheorem{cor}[teo]{Corollary}
    \newtheorem{rem}[teo]{Remark} 

\usepackage{comment}
\usepackage{float} 
\def\ogg~{{\rm \og}}   

\def\emptyset{\varnothing}

\makeatletter
\def\BState{\State\hskip-\ALG@thistlm}
\makeatother

\title{A Grothendieck ring of finite characteristic}

\address{Department of mathematics, Ben Gurion University of the Negev}
\email{Esther.Elbaz.Saban@normalesup.org}

\author{Esther Elbaz}

\date{\today }
\begin{document}

\maketitle

\begin{abstract}
We construct, for every integer $N\in \N^*$, a structure whose Grothendieck ring is isomorphic to $(\Z/N\Z)[X]$, thus proving the existence of structures with a non-zero Grothendieck ring with non-zero characteristic. Namely, this structure consists in the bijection without cycles between a set and a complement of $N$ points in this set.
\end{abstract}

\section{Introduction}\label{s.introduction}

Originating from algebraic geometry, and inspired by the work of Schanuel \cite{Sch}, Grothendieck rings were introduced in the larger framework of model theory in 2000 by T. Scanlon and J. Krajicek \cite{SK} on the one hand and F. Loeser an J. Denef \cite{LD} on the other hand. Roughly speaking, to each structure, $M$, one can associate its Grothendieck ring by identifying its definable subsets of a Cartesian power of $M$, that are in definable bijection and then endowing the quotient set with an additive law that corresponds on the level of
definable subsets to the disjoint union, and a multiplicative law that corresponds on the
level of definable subsets to the product.


To give the precise definition, let us start by defining the Grothendieck semi-ring. 
We recall that a semi-ring is an algebraic structure similar to a ring with the difference that the additive law need not be cancellative. More precisely, $(R, +, \times)$ is a \textbf{semi-ring} iff 
\begin{itemize}
\item $(R, +)$ is a commmutative monoid with neutral element $0$,
\item $(R, \times)$ is a monoid,
\item $0$ annihilates $R$
\item the law $\times$ is distributive (on the right and on the left) over the law $+$.
\end{itemize}

\begin{de}\label{GR}
Let $M$ be a non-empty structure and let $\mathrm{Def}(M)$ be the set of definable (with parameters) subsets of any finite Cartesian power of $M$.
We define an equivalence relation on $\mathrm{Def}(M)$ by: $A, B\in \mathrm{Def}(M)$ are equivalent if and only if, there exists a definable bijection between them.\\
We consider the quotient set, denoted $ \widetilde{\operatorname {Def}}(M) $, and denote by $[A]$ the class of $A\in \mathrm{Def}(M)$. We define two laws on $ \widetilde{\operatorname {Def}}(M) $:
\begin{itemize}
\item An additive law that corresponds to the disjoint union: $[A]+[B]=[A\cap B]+[A\cup B]$.
\item A multiplicative law that corresponds to the Cartesian product: $[A]\times[B]=[A\times B]$.
\end{itemize}
These two laws are compatible with the equivalence relation and the set $ \widetilde{\operatorname {Def}}(M) $ endowed with these two laws is the Grothendieck semi-ring of $M$.
\end{de}

Beware that the additive law is not cancelative: $a+b=a+c$ doesn't imply $b=c$. To remedy this, let us consider the equivalence relation defined on $ \widetilde {\operatorname {Def}}(M) $ by: $a, b\in \widetilde {\operatorname {Def}}(M) $ are equivalent if and only if there exists $c\in \widetilde {\operatorname {Def}}(M) $ such that $a+c=b+c$. The quotient set, denoted $\widetilde {\operatorname {Def}}'(M)$, is a cancelative monoid for the additive law. Hence, there exists a unique ring (up to isomorphism over $(\widetilde{\operatorname{Def}}'(M), +, \times) $) that embeds this last quotient and that is minimal for this property.

\begin{de}
The minimal ring that embeds $(\widetilde{{\operatorname {Def}}}'(M), +, \times) $ is called the \textbf{Grothendieck ring of} $M$ and is denoted $K_0(M)$.
\end{de}

Two quotients naturally occur in this definition. The first one identifies subsets in definable bijection, the second makes the additive law cancelative. This second quotient is at the origin of collapse phenomena. Let us take the example of the real field: $ \R $ is equal to the disjoint union $ \R^{<0} \sqcup \lbrace 0 \rbrace \sqcup \R^{> 0} $ and, in the language of fields, both $ \R^{> 0}$ and $\R^{<0} $ are in definable bijection with $ \R $. Denoting $ [\R] $ the class of $ \R $ in its Grothendieck ring, we obtain that $ 2 [\R] + 1 = [\R] $ which implies that $ [\R] = - 1 $: the class of $ \R $ is equal to the opposite of the class of a singleton. The Grothendieck ring of the real closed field $ \R $ thus does not account for the dimension of definable subsets. It is actually isomorphic to $\Z$ \cite{SK}.

The onto-pigeonhole principle, abbreviated onto-PHP, states that the Grothendieck ring of a structure is non-zero if and only if none of its
definable subsets is in definable bijection with itself deprived of a point. 
This equivalence has been demonstrated by Krajicek and Scanlon \cite{SK} and was used to show the vanishing of the 
Grothendieck ring of several rings and fields, such as for example, those 
of Laurent series fields \cite{Cl}, or more generally of certain discretely valued fields \cite{Cl, HC}. 
Beyond the case of fields which is by far the most studied, there are few structures whose Grothendieck ring is known. Schanuel considers a suitable reduct of the field $\R$ for which definable subsets are polyhedra, and obtains the Grothendieck ring $\Z$. For a suitable variant, Ma\v r\'ikov\'a \cite{Marikova} obtains the ring $\Z[X]/(X+X^2)$.
 A scheme of computation for an arbitrary module over a unital ring was obtained by Kuber \cite{Kuber}; in particular for the $p$-adic ring $\Z_p$, viewed as module over itself, or as plain abelian group, it is shown \cite[Cor.\ 6.2.3]{Kuber} to be isomorphic to $\Z[X]/((p-1)X)$; these are (for odd prime $p$) the first examples for which the underlying abelian group is not torsion-free.
Nevertheless, it was unknown whether there exists a structure with a non-zero Grothendieck ring of finite characteristic (i.e., whose additive group is nontrivial and torsion). In this paper, we answer this question positively by constructing, for any integer $N\in \N_{>0}$, a structure whose Grothendieck ring is $(\Z/N\Z)[X]$.

\begin{teo}\label{but}
For any positive integer $N\in \N$, there exists a structure whose Grothendieck ring is isomorphic to $(\Z/N\Z)[X]$.
\end{teo}

Let $M$ be a candidate structure and let us denote $X$ the class of $M$ in the Grothendieck ring. 
Because we want the Grothendieck ring to be of characteristic $N$, it is natural to assume $M$ to be endowed with a bijection, $f$, with itself deprived of $N$ points. Indeed, such a bijection implies that, in $K_0(M)$, we have the equality $X-N=X$ and hence that $N=0$. The characteristic of the ring $K_0(M)$ is thus a divisor of $N$.

To conclude that the Grothendieck ring is precisely $(\Z/N\Z)[X]$, we will rely on a crucial property: the existence of a class of definable sets, called "simple sets" satisfying the following properties:
\begin{itemize}
\item (P1) any definable subset can
be built as a Boolean combination of simple sets;
\item (P2) they are in definable bijection with a Cartesian power of $M$;
\item (P3) injections whose graph is defined by a conjunction of equalities between terms, transform a simple set into another simple set.
\end{itemize}
We will show that to obtain these properties it is sufficient to require the bijection $f$ to be without cycles.
Properties P1 and P2 then imply that $K_0(M)$ is a quotient of $\Z[X]$. We already know, because of the existence of $f$, that $K_0(M)$ is of characteristic $N$ so we can obtain that $K_0(M)$ is a quotient of $(\Z/N\Z)[X]$. 
To conclude, we have to verify there are no other independent relations in $K_0(M)$.
The property P3 and the properties of simple sets that can be seen as irreducible closed sets of the noetherian topology they generate, are the final ingredient of the proof.

It should be noted that in the general case two elementarily equivalent structures do not necessarily have isomorphic Grothendieck rings. However, Perera showed it holds in the case of modules over a commutative ring~\cite{Ps}.

Here we prove that all the model of the bijection without cycles with a complement of $N$ points have the same Grothendieck ring, isomorphic to $(\Z/N\Z)[X]$. In particular, Theorem \ref{but} results from the following theorem:

\begin{teo}\label{but2}
Let $N\in \N_{>0}$, let $L_N:=\lbrace f\rbrace$ be a language consisting of a unary function symbol and let $\mathcal{T}_N$ be the theory
stating that:
\begin{itemize}
\item the function $f$ is a bijection with coimage a finite set of cardinality $N$;
\item for every $n\in \N_{>0},$ $f^n(x)\neq x$ where $f^n$ is the $n$-th iterate of $f$.
\end{itemize}
Then any model of $ \mathcal{T}_N$, $M$, admits a Grothendieck ring isomorphic to $(\Z/N\Z)[X]$. Furthermore, the class of $M$ in $K_0(M)$ is $X$.
\end{teo}

\begin{cor}
Let $N\in \N_{>0}$ and $ \mathcal{T}_N$ be the theory defined in Theorem \ref{but2} and $M$ a model of $ \mathcal{T}_N$.
Then, for every $k\in \N$, there exists a definable injection from $M$ into $M$ minus $k$ points if and only if $N$ divides $k$. In particular (for $N\ge 2$), we have a failure of the "Cantor-Bernstein property": for any point $x\in M$ there exist definable injections 
$M\rightleftarrows M\smallsetminus \{x\}$ but no definable bijection between these two sets.\qed
\end{cor}


\begin{rem}
As a corollary, we also obtain, that the "dimension semi-ring" (\cite{Sch}, see Definition \ref{de_dimsemi}) of any model of $\mathcal{T}_N$ is isomorphic to $\{-\infty\}\cup\N$, with addition being max and multiplication being usual multiplication (Corollary \ref{cor_dimsem}).
\end{rem}

Some of the ideas we use here to prove Theorem \ref{but} ---~ like the construction of a theory whose models are the freest possible structures admitting the prechosen ring as their Grothendieck ring, or like the possibility of defining a class of "simple sets" with properties (P1), (P2), (P3)~--- will be used in a forthcoming work to construct, given any quotient $R$ of $\Z[X]$, a structure admitting $R$ as its Grothendieck ring.

In this paper, after defining the theory, we compute the Grothendieck ring of any of its model by the general method explained above.

\tableofcontents

\section{The theory}
We fix an integer $N\in \N_{>0}$.
The purpose of the rest of this paper is to construct a structure with Grothendieck ring isomorphic to $\Z[X]/(N)$.
Sticking to the general ideas explained above, we place ourselves in a theory stating the existence of a bijection between $M$ and itself deprived of $N$ points. 
In order to be able to define the simple sets mentioned in the previous section, we require furthermore that $f$ has no cycles i. e. for any integer $n\in\N^*$, the function $f^n$ has no fixed points. 

More precisely, let $L$ be the language $\lbrace f, c_1,\ldots, c_N\rbrace$ where $f$ is a unary function, and for every $1\leq i\leq N$, $c_i$ is a symbol of constant.
Let $\mathcal{T}_N$ be the first-order theory stating that:
\begin{itemize}
\item for every $i, j\in \lbrace 1,\ldots N\rbrace$, $i\neq j$ implies $c_i\neq c_j$;
\item $f$ is a bijection from $M$ onto $M\smallsetminus \lbrace c_1,\ldots, c_N\rbrace$;
\item for every integer $n\in \N_{>0}$, $\forall x\in M$, $f^{n}(x)\neq x$ where $f^n$ is the $n$-th iterate of $f$.
\end{itemize}

It is trivial to see that this theory is consistent. For instance, $\N$ endowed with the function $f:x\mapsto x+N$ and $c_1=0, \ldots, c_N=N-1$ is a model of $\mathcal{T}_N$.

For the sequel of this paper, we use the convention that $f^0$ is the identity function.

\section{Quantifier elimination}

\begin{prop}\label{prop_qe}
The theory $\mathcal{T}_N$ admits quantifier elimination.
\end{prop}

\begin{proof}
We use the back-and-forth method.
Let $M$ and $M'$ be two saturated models of $T$.
Let $\mathbf{a}$ be an $l$-tuple of elements of $M$ and $\mathbf{a'}$ be an $l$-tuple of elements of $M'$ that has the same quantifier-free type as $\mathbf{a}$.
Let $b\in M$. Let us find $b'$ such that $\mathbf{a}^{\wedge}b$ and $\mathbf{a'}^{\wedge}b'$ have the same quantifier-free type.
If $b$ is in the orbit of $\lbrace c_1, \ldots, c_N\rbrace$, then it is sufficient to take $b'$ equal to the . In the following, we assume that $b\notin \lbrace c_1, \ldots, c_N\rbrace$.
The other quantifier-free formulas satisfied by $b$ over $\textbf{a}$ are of the form:
\begin{itemize}
\item $f^s(b)=f^t(a_i)$ 
\item $f^s(b)\neq f^t(a_i)$
\end{itemize}
where $1\leq i\leq l$ and $s, t\in \N$.
As $f$ has no cycles, there is no formula of the form $f^s(b)=f^t(b)$ with $s\neq t$.

Let $n\in \N$.
By compactness, it is enough to show that we can satisfy these formulas for $s, t$ bounded by $n$.

Since $f$ is a bijection, any formula of the form $f^s(x)=f^t(y)$ with $s, t\leq n$, is equivalent to a formula 
$f^k(x)=y$ or a formula $f^k(y)=x$ where $k\leq n$.
If for some $i, k$, $f^k(b)=a_i$ (respectively $f^k(a_i)=b$), then $a_i\notin \lbrace c_1, \ldots, c_N\rbrace$. Since $\mathbf{a}$ and $\mathbf{a'}$ have the same quantifier-free type, $a'_i\notin \lbrace c_1, \ldots, c_N\rbrace$. We take $b'$ to be the unique element $b'$ of $M'$ that satisfies $f^k(b')=a'_i$ (resp.\ $f^k(a'_i)=b'$): such an element exists and is unique since $f$ is a bijection from $M$ to $M\setminus \lbrace c_1, \ldots, c_N\rbrace$. 
If for every $1\leq i\leq l$, and every integer $k$, $f^k(b)\neq a_i$ and $f^k(a_i)\neq b$, then we take $b'$ to be any element of $M'\smallsetminus \lbrace f^j(a'_i)| -n\leq j\leq n, 1\leq i\leq l\rbrace$.
Any model of $\mathcal{T}_N$ being of course infinite, this is always possible.

By construction $\mathbf{a}^{\wedge}b$ and $\mathbf{a'}^{\wedge}b'$ have the same quantifier-free type.
\end{proof}

\section{Simple sets}\label{section_simple}

For all the rest of this paper, we consider $M$ a model of $\mathcal{T}_N$. It is generally understood that definable sets are subsets of some Cartesian power of $M$.

In this section, we will consider a subclass of the class of definable sets which we call "simple sets" and which will have the following three
properties:
\begin{enumerate}
\item Every simple set is in definable bijection with a Cartesian power of $M$.
\item The topology whose family of closed sets is generated by the simple sets is Noetherian and its irreducible closed sets are the simple sets.
\item Every definable set is a Boolean combination of simple sets.
\end{enumerate}

\begin{de}
Let $n\in \N_{>0}$.
A subset $A\subseteq M^n$ is said to be \textbf{simple} if it is definable by a conjunction of formulas of the form:
\begin{itemize}
\item $x_{i}=f^k(x_{j})$ where  $i, j, k\in \N$, $1\leq i,j \leq n$ and $\sigma$ is a permutation of $[\![ 1;n]\!] $; 
\item $x_i=c$ where $c\in M$ and $1\leq i\leq n$.
\end{itemize}
\end{de}

\begin{rem}\label{intersection_simples}
The intersection of two simple sets is a simple set. No simple set is the finite union of proper subsets that are simple.
\end{rem}
\begin{de}
Let $A\subseteq M^n$ be a simple set.
We define the set $C:=\lbrace i\in [\![ 1;n]\!]| \exists a\in M, \forall x\in A, x_i=a\rbrace$
and on $[\![ 1;n]\!] \smallsetminus C$, the equivalence relation : $i$ is equivalent to $j$ if and only if there exists $k\in \N$ such that, for every $x\in A$, $x_i=f^k(x_j)$ or $x_j=f^k(x_i)$. Let $I_1,\ldots, I_d$ be the equivalence classes.

Every equivalence class, $I$, contains an element $i$ (not necessarily unique) such that for every $j\in I$, there exists a positive integer $k\in \N$ such that $x_j=f^k(x_i)$. We call such an element a \textbf{determining element} of $I$.
We say that $C\sqcup (\bigsqcup_{i=1}^d I_i)$ is the \textbf{partition of $[\![ 1;n]\!] $ associated to $A$}, that $I_1,\ldots, I_d$ are the \textbf{equivalence classes of $A$} and that $C$ is the \textbf{set of constants} of $A$.
\end{de}

\begin{lem}\label{incl}
Let $A$ be a simple set with set of constants $C$ and equivalence classes $I_1,\ldots, I_d$.
Let $B$ be a simple set with set of constants $C'$.
If $A\subsetneq B$, then
\begin{itemize}
\item $C'\subseteq C$;
\item For every $1\leq j\leq d$, $I_j$ is a union of some of the equivalence classes of $B$;
\item $d<d'$.
\end{itemize}
\end{lem}

\begin{proof}
The first two properties are obvious.
If the inclusion $C'\subseteq C$ is strict, then it is clear that the number of equivalence classes of $A$ is strictly less than that of $B$.
If $C'=C$, then there exists $1\leq j\leq d'$ such that $I_j$ is the union of at least two equivalence classes of $B$. This implies that $d$ is strictly less than $d'$. Thus, in all cases, the number of equivalence classes of $A$ is strictly less than that of $B$.
\end{proof}

\begin{lem}
Every simple set is in definable bijection with some Cartesian power of $M$. More precisely if $A$ admits $d$ equivalence classes, then $A$ is in definable bijection with $M^d$.
\end{lem}

\begin{proof}
Let $A\subseteq M^n$ be a simple set.
Let $C\sqcup (\bigsqcup_{i=1}^d I_i)$ be the partition of $[\![ 1;n]\!] $ associated to $A$ and for each $ 1 \leq i \leq d$,
let $j_i$ be a determining element of $I_i$. Then for every $(x_1,\ldots, x_d) \in M^d$, there exists a unique element $(a_1,\ldots, a_n)$ of $A$ such that for every $1 \leq i\leq d$, one has $a_{j_i} = x_i$. The function that associates to every element of $M^d$ the unique element of $A$ satisfying these equalities is a definable bijection between $M^d$ and $A$.
\end{proof}

\begin{de}\label{dim}
Let $A\subseteq M^n$ be a simple set that admits $d$ equivalence classes. Then we say that $d$ is the \textbf{dimension} of $A$ and denote it by $\dim(A)$.
\\
Let $k\in \N^*$ and $X_1,\ldots, X_k$ be simple sets of respective dimensions $d_1,\ldots, d_k$.
Let $A:= \cup_{i=1}^k X_i$. Then we say that $\max_{i=1}^k d_i$ is the \textbf{dimension} of $A$ and denote it $\dim(A)$.
\end{de}

\begin{rem}\label{dimaj}~\\
\textbf{a}. It is immediate to notice that if $A$ and $B$ are two simple sets such that $A\subsetneq B$, then $\dim(A)<\dim(B)$.
\\
\textbf{b}. The dimension of $M^n$ is $n$.
\end{rem}

We prove below that for every integer $n\in \N_{>0}$, the simple subsets of $M^n$ are the irreducible closed sets of the topology on $M^n$ that they generate. Furthermore, this topology is Noetherian and the dimension defined above coincides with Krull dimension.
Let us first recall some basic definitions and results.

\begin{de}
Let $X$ be a topological space. We say that the topology on $X$ is \textbf{Noetherian} if 
there is no infinite descending chain of closed sets.
A closed subset is said \textbf{irreducible} if it is nonempty and cannot be written as the union of two proper closed subsets.
\end{de}

The following lemma is classical.

\begin{lem}
Let $X$ be a Noetherian space and let $F \subseteq X$ be closed. Then there are 
irreducible closed sets $F_1,\ldots, F_n$ such that $F = \bigcup_i F_i$. Moreover, this decomposition is unique up 
to permutation if there is no inclusion between the $F_i$.\qed
\end{lem}

\begin{de}
The subsets $F_i$ of the previous lemma are called the\textbf{ irreducible components} of $F$ and $F = \bigcup_i F_i$ is called the \textbf{decomposition into irreducible components} of~$F$.
\end{de}

\begin{de}
Let $X$ be a non-empty closed subset of a Noetherian topology.
The \textbf{dimension} of $X$ is defined by:
$$\dim_{top} (X)=\sup\lbrace n\in \N \mid \emptyset=Z_{-1}\subsetneq Z_0\subsetneq \ldots\subsetneq Z_n \subseteq X \text{ is a chain of irreducible closed sets} \rbrace.$$

By convention the dimension of the empty set is $-\infty$.
\end{de}

\begin{lem}\label{noe}
Let $m, n\in \N^*$ and for every $1\leq i\leq m$ (respectiveley every $1\leq i \leq n$) $X_i$ (respectively $Y_i$) a simple set.
Let us assume that $\cup_{i=1}^m X_i \subseteq \cup_{i=1}^n Y_i$.
Then for every $1\leq i\leq m$, there exists $1\leq \sigma(i)\leq n$ such that $X_i\subseteq Y_{\sigma(i)}$.
\end{lem}

\begin{proof}
Let $1\leq j\leq m$, then $X_j = \cup_{i=1}^n (X_j\cap Y_i)$.
Since the intersection of two simple sets is simple and since a simple set cannot be the finite union of proper subsets that are simple, this equality implies that there exists $1\leq \sigma(j)\leq n$ such that $X_j= X_j\cap Y_{\sigma(j)}$.
\end{proof}

\begin{lem}\label{topology_noeth}
Let $n\in \N_{>0}$.
The topology $\Top_n$ on $M^n$ whose family of closed sets is generated by the simple subsets of $M^n$ is Noetherian and its irreducible closed sets are the simple sets. 
\end{lem}

\begin{proof}
We start by a definition:
\begin{de}
Let $A$ be a closed set. Let $k\in \N^*$ and $X_1,\ldots, X_k$ be irreductible sets such that $A=\cup_{i=1}^k X_i$ and for every $i\neq j$, $X_i$ is not included in $X_j$. 
We associate to $A$, the element $s(A)$ of $\N[X]$ defined by $s(A):=\sum_{i=1}^k X^{\dim(X_i)}$.
\end{de}

We define on $\N[X]$ the natural order induced by the usual order on $\N$ and the relations, for every $i, j\in \N$, $X^i\leq X^j \Leftrightarrow i\leq j$. It is immediate to notice that any decreasing sequence of elements of $\N[X]$ is stationary.

\medskip
Let us first prove that if $A$ and $B$ are two closed sets, then $A\subsetneq B$, if and only if $s(A)< s(B)$.

Let $X_1,\ldots, X_k$ (respectively $Y_1,\ldots, Y_l$) be such that $A=\cup_{i=1}^k X_i$ (respectively $B=\cup_{i=1}^l Y_i$).
Since $A\subsetneq B$, by lemma \ref{noe}, for every $1\leq i\leq k$, there exists $1\leq \sigma(i)\leq l$ such that $X_i\subseteq Y_{\sigma(i)}$. Furthermore, either one of these inequalities is strict or $k<l$. By point \textbf{a} of remark \ref{dimaj}, if $X$ and $Y$ are two simple sets such that $X\subsetneq Y$, then $\dim(X)<\dim(Y)$. As a result, $s(A)< s(B)$.

\medskip
Let us now show that any decreasing chain of closed sets is stationary.
Let $(A_i)_{i\in \N}$ be such a chain. The sequence $(s(A_i))_{i\in \N}$ is a decreasing sequence of $\N[X]$. It is thus stationary. Hence so is the sequence $(A_i)_{i\in \N}$.

It remains to prove that the irreductible closed sets are the simple sets. It is obvious by definition that an irreductible closed set is a simple set. Converserly, as we have noticed in remark \ref{intersection_simples}, a simple set cannot be the finite union of proper subsets that are simple.
\end{proof}

For every integer $n\in \N_{>0}$, the dimension of a simple set as a closed set of $\Top_n$ coincides with the dimension defined above:

\begin{lem}\label{lem_defbijmd}
Let $n\in \N_{>0}$. 
Let $A\subseteq M^n$ be an infinite simple set whose dimension as a closed subset of $\Top_n$ is $d$. Then $A$ is in definable bijection with $M^d$. As a result $\dim(A)=\dim_{top}(A)$.
\end{lem}

\begin{proof}
Let $\emptyset=Z_{-1}\subsetneq Z_0\subsetneq \ldots\subsetneq Z_d =A$ be a chain of simple sets.
Then for every $-1 \leq i \leq d-1$, 
$Z_{i}$ has strictly fewer many equivalence classes than $Z_{i+1}$. 
A chain of simple sets included in $A$ thus has a length bounded by the cardinality of the partition associated to $A$.
It is obvious that a chain of this length exists. The dimension of $A$ as a closed set of $\Top_n$ is thus equal to its dimension as defined above.
\end{proof}

\begin{de}
Let $A\subseteq M^n$ be a definable subset and let $\bar{A}$ be its closure in $\Top_n$. We define $\dim(A):=\dim(\bar{A})$.
\end{de}

\begin{prop}\label{dim_inclu}
Let $A\subseteq M^n$ be a nonempty definable subset of dimension $d$. Then there exist nonempty finite subsets $F,F'$ of $M$ and definable injections $M^d\to F\times A$ and $A\to F'\times M^d$. 
\end{prop}

\begin{proof}
Because $A$ has dimension $d$, its closure contains a (simple) subset $B$, in definable bijection with $M^d$. Thus $A$ contains a subset of the form 
$B\smallsetminus\bigcup_{i=1}^k B_i$, where all $B_i$ are simple sets of dimension $\le d-1$. Therefore, by lemma \ref{lem_defbijmd}, $M^d\smallsetminus (F\times M^{d-1})$ with $|F|=k$ embeds in $A$ via a definable injection, say $h_1$. Let us notice that $M^{d-1}$ embeds definably into $M^d\smallsetminus (F\times M^{d-1})$, and hence into $A$. Call $h_2$ this last injection from $M^d\smallsetminus (F\times M^{d-1})$ to $A$. Take $b_in M\setminus F$. Define the injection $g:M^d\to (F\sqcup\{b\})\times A$ by: $g(x)=(b,h_1(x))$ for $x\in M^d\smallsetminus (F\times M^{d-1})$, and for $x=(q,y)\in F\times M^{d-1}$, define $g(x)=(q,h_2(y))$.

For the second injection, we can suppose $A$ to be closed. So $A$ is a finite union of $\ell$ simple sets of dimension $\le d$, each of which being, by Lemma \ref{lem_defbijmd}, in definable bijection with $M^{d'}$ for some $d'\le d$. Hence $A$ embeds definably into $F'\times M^d$ with $|F'|=\ell$. 
\end{proof}

In \cite{Sch}, Schanuel defines the Burnside rig (commutative ring without negatives) associated to a distributive category, $\mathcal{C}$, in a very similar way that one defines the Grothendieck semi-ring of a structure. It consists of the isomorphism classes of objects of $\mathcal{C}$, endowed with the additive law induced by the coproduct and the multiplicative law induced by the product. He computes the Burnside rig of some categories (namely the category of bounded polyhedra, the category of unbounded polyhedra, the category of semi-algebraic sets).
To do so he proves that every isomorphism class of the specific categories he works in is entirely determined by its image under two maps:
\begin{itemize}
\item its image in the Grothendieck ring (though he does not call it by that name) and
\item its image in the dimension ring defined as below.
\end{itemize}
In \cite{SK}, Scanlon and Krajicek computed, in a very similar way, the Grothendieck ring of real closed fields to be isomorphic to $\Z$.

\begin{de}\label{de_dimsemi}The \textbf{dimension semi-ring} \cite{Sch} of a structure $M$ is defined as follows: for $A,B\in \widetilde{\operatorname {Def}}(M)$, write $A\preceq B$ if there exists a definable injection $A\to B$. Write $A\precsim B$ if there exists $n\in\N$ such that $A\preceq nB$. Finally the dimension semi-ring of $M$ is the quotient of $\widetilde{\operatorname {Def}}(M)$ by the relation identifying $A$ and $B$ if $A\precsim B\precsim A$. 
\end{de}

\begin{cor}
\label{cor_dimsem}
The dimension semi-ring of our structure $M$ is isomorphic to $(\{-\infty\}\cup\N,\max,+)$, the isomorphism being induced by mapping a definable set to its dimension.
\end{cor}
\begin{proof}
Let $A\subseteq M^n$ and $B\subseteq M^m$ be definable subsets, of dimension $d,d'$.

We have to show that $d\le d'$ if and only if there exist a finite set $F$ and a definable injection $A\to F\times B$. We can suppose that $d,d'\ge 0$.

If $d\le d'$, by Proposition \ref{dim_inclu} there exist finite subsets $F_1, F_2$ of $M$ and definable injections $A\to F_1\times M^d$ and $F_2\times M^{d'}\to B$. So $F_1\times F_2\times M^d$ embeds definably into $F_1\times B$, and hence by composition $A$ embeds definably into $F_1\times B$.

Conversely if $A$ embeds definably into $F\times B$, suppose for contradiction that $d>d'$. By Proposition \ref{dim_inclu}, there exist definable injections $B\to F_1\times M^{d-1}$ and $M^d\to F_2\times A$. By composing suitably, we obtain a definable injection from $M^d$ into $F_1\times F_2\times F\times M^{d-1}$. This contradicts the fact that $M^d$ has dimension $d$.
\end{proof}

\section{Decomposition of definable sets in simple sets}

By quantifiers eliminations, every definable set is a Boolean combination of simple sets.
We refine this in the following lemma.

\begin{lem}\label{lem_boolsim}
Every definable set is a Boolean combination of simple sets. More precisely, if $A$ is a definable set, then there exists a finite Boolean combination, $\bigsqcup_{i\in I} (B_i\smallsetminus (\bigcup_{j\in I_i} B_{i,j}))$ such that:
\begin{itemize}
\item for every $i, j$, $B_{i,j}\subsetneq B_i$ are simple sets;
\item $A$ is equal to the disjoint union $\bigsqcup_{i\in I} (B_i\smallsetminus \bigcup_{j\in I_i} B_{i,j})$.
\end{itemize}
\end{lem}

\begin{proof}
Because the intersection of simple sets is a simple set, it is an immediate consequence of quantifier elimination (Proposition \ref{prop_qe}) that we can write $A=\bigcup_{i\in I} (B_i\smallsetminus \bigcup_{j\in I_i} B_{i,j})$ such that for every $i, j$, $B_{i,j}, B_i$ are simple sets and $B_{i,j}\subsetneq B_i$.

Let us show that we can modify the union indexed by $I$ in order to obtain a disjoint union. If there exist $ i, i'$ such that $(B_i \smallsetminus \bigcup_{j\in I_i}  B_{ i,j}) \cap (B_{i'} \smallsetminus \bigcup_{j\in I_{i'}}  B_{i', j })\neq \emptyset $, then $ B_i \cap B_{i'} \neq \emptyset $. Let $ C: = B_i \cap B_{i'} $. The set $C$ is simple and $ (B_i \smallsetminus (\bigcup_{j\in I_i}  B_{ i,j} \cup C)) \cap (B_{i'} \smallsetminus \bigcup_{j\in I_{i'}}  B_{i', j }) = \emptyset $.
Let us notice that
\begin{align*}&(B_i \smallsetminus \bigcup_{j\in I_i}  B_{ i,j})\cup (B_{i'} \smallsetminus \bigcup_{j\in I_{i'}}  B_{i', j })\\
&=(B_i \smallsetminus (\bigcup_{j\in I_i}  B_{ i,j} \cup C))\cup (B_{i'} \smallsetminus \bigcup_{j\in I_{i'}}  B_{i', j })\cup (C\cap (B_i\smallsetminus \bigcup_{j\in I_i}  B_{ i,j}))\\
&=(B_i\smallsetminus (\bigcup_{j\in I_i}  B_{ i,j} \cup C))\cup (B_{i'} \smallsetminus \bigcup_{j\in I_{i'}}  B_{i', j }) \cup ((C\cap B_i)\smallsetminus  \bigcup_{j\in I_{i'}}  (C\cap B_{i', j })).
\end{align*}

We can thus modify the partition as follows:
\begin{itemize}
\item we replace $B_{i} \smallsetminus (\bigcup_{j\in I_{i}}  B_{i, j })$ by $B_i\smallsetminus (\bigcup_{j\in I_i}  B_{ i,j} \cup C)$;
\item we add $C\smallsetminus  (\bigcup_{j\in I_{i'}} C\cap B_{i', j })$.
\end{itemize}

This construction cannot be repeated indefinitely as there is no infinite chain of simple sets.
By repeating this as many times as necessary,
we obtain a disjoint union.
\end{proof}

Consider the natural homomorphism $\tilde{\pi}$ from the polynomial ring $\Z[X]$ to $K_0(M)$ mapping $X$ to the class $[M]$ of $M$ and $1$ to the class of the singleton. Since $M$ is in definable bijection with $M$ deprived of $N$ points, we have $N=0$ in $K_0(M)$ and therefore $\tilde{\pi}$ factors through a ring homomorphism \[\pi:(\Z/N\Z)[X]\to K_0(M).\]

Thanks to the previous lemma, to prove that $\pi$ is surjective it is enough to prove that for any simple set $B$ and any simple sets $B_1,\ldots, B_n$ included in $B$, the class of $B\setminus \cup_{i=1}^n B_i$ is a Boolean combination of $[B]$ and of the $[B_i]$.

This is the point of the next lemma. We start by defining specific representants of a class in $K_0(M)$.

\begin{de}
Let $n\in \N$. We say that a subset $A$ of $M^{n+1}$ is \textbf{basic} if it is a singleton or if, for every $0\leq k\leq n$, there exist a singleton $S_k \subseteq M^{n-k}$ and a finite set $F_k\subsetneq M$ (possibly empty) such that
 $$A:=\bigsqcup_{k=0}^n F_k \times M^k \times S_k.$$ 
 
We define its \textbf{associated polynom} as $P(X):=\sum_{k=0}^d |F_k| X^k$ where for evey $0\leq k\leq n$, $|F_k|$ is the cardinality of $F_k$.

\medskip
Conversely, if $P(X):=\sum_{k=0}^d a_k X^k \in \N[X]$, then we call\textbf{ basic set associated to $P(X)$} any basic set of the form $A:=\bigsqcup_{k=0}^d F_k \times M^k \times S_k$ where $F_k$ is a finite subset of $M$ of cardinality $a_k$ and $S_k$ a singleton of $M^{d-k}$.
\end{de}
 
\begin{lem}\label{triv}
Let $A=\bigsqcup_{i\in I}(B_i\smallsetminus \bigcup_{j\in I_i} B_{i,j})$ where 
\begin{itemize}
\item for every $i\in I$, $B_i$ is a simple set,
\item for every $j\in I_i$, $B_{i,j}\subsetneq B_i$ is a simple set.
\end{itemize}
\begin{enumerate}[label=(\alph*)]
\item\label{item3_a} Then, following notations of definition \ref{GR}, $[A]$ is a linear combination over $\Z$ of 
the $[B_i]$ and of the $[\bigcap_{j\in K} B_{i,j}]$ where $i\in I$ and for every $i\in I$, $K$ ranges over the subsets of $I_i$.

\item\label{item3_b} In particular, for any definable set $A$, there exist $N$ and $P$, basic sets disjoint from $A$ and of dimension at most $\dim(A)$, such that $A\sqcup N$ is in definable bijection with $P$.
Furthermore, for any definable set $X$, we can choose $P$ and $N$ to be disjoint from $X$.
\end{enumerate}
\end{lem}
\begin{proof}
\ref{item3_a} Let us first prove the result in the case where for every $i\in I$, the set $I_i$ is empty. We proceed by induction on the cardinality of $I$.
The case where $I$ has cardinality $1$ is trivial.

Let $n\in \N^*$ and let us assume that the result is proved for any set $I$ of cardinality at most $n$.
By definition, $[\bigcup_{i=1}^{n+1} B_i]=[B_{n+1}]+[\bigcup_{i=1}^{n} B_i]-[\bigcup_{i=1}^{n} (B_{n+1}\cap B_i)].$
Since the intersection of two simple sets is still a simple set, the induction hypothesis applies on $[\bigcup_{i=1}^{n} B_i]$ and $[\bigcup_{i=1}^{n} (B_{n+1}\cap B_i)]$ and allows us to conclude.

\medskip
\noindent
Now, let us prove the general result by induction on the cardinality of $I$.
Let us assume that $I$ has only one element that we denote $i$.
By definition, $[B\smallsetminus \bigcup_{j\in I_i} B_{i,j}]$ is equal to $[B] - [ \bigcup_{j\in I_i} B_{i,j}]$. By what preceeds $[ \bigcup_{j\in I_i} B_{i,j}]$ is a linear combination over $\Z$ of 
the $[\bigcap_{j\in K} B_{i,j}]$ where $K$ ranges over the subsets of $I_i$. 

Let $n\in \N^*$.
Let us assume that the result is true for any set of cardinality $n$.
By definition, $[\bigsqcup_{i=1}^{n+1}(B_i\smallsetminus \bigcup_{j\in I_i} B_{i,j})]$ is equal to $[\bigsqcup_{i=1}^{n}(B_i\smallsetminus \bigcup_{j\in I_i} B_{i,j})]+[B_{n+1}\smallsetminus \bigcup_{j\in I_{n+1}} B_{i,j}]-
[B_{n+1}\smallsetminus  (\bigcup_{j\in I_{n+1}} B_{i,j} \cap \bigsqcup_{i=1}^{n}(B_i\smallsetminus \bigcup_{j\in I_i} B_{i,j}))]$.

Let us notice that \\
$(B_{n+1}\smallsetminus \bigcup_{j\in I_{n+1}} B_{i,j})\cap \bigsqcup_{i=1}^{n}(B_i\smallsetminus \bigcup_{j\in I_i} B_{i,j})\\
=\bigsqcup_{i=1}^{n}((B_i\smallsetminus \bigcup_{j\in I_i} B_{i,j}) \cap (B_{n+1}\smallsetminus \bigcup_{j\in I_{n+1}} B_{i,j}))\\
=\bigsqcup_{i=1}^{n}((B_{n+1}\cap B_i)\smallsetminus \bigcup_{j\in I_i} (B_{i,j}\cup (B_i\cap \bigcup_{j\in I_{n+1} } B_{n+1,j})))
.$

The induction hypothesis applies on $[\bigsqcup_{i=1}^{n}((B_{n+1}\cap B_i)\smallsetminus (\bigcup_{j\in I_i} B_{i,j}\cup (B_i\cap \cup_{j\in I_{n+1} } B_{n+1,j})))]$ and on $[\bigsqcup_{i=1}^{n}(B_i\smallsetminus \bigcup_{j\in I_i} B_{i,j})]$. We hence just have to notice that $[B_{n+1}\smallsetminus \bigcup_{j\in I_{n+1}} B_{i,j}]=[B_{n+1}] -[\bigcup_{j\in I_{n+1}} B_{i,j}]$ and use the result already proven for the union of simple sets.

\ref{item3_b} By what preceeds the class of any definable set $A$ is of the form $\sum_{i=0}^{\dim(A)} a_i T^i$ where $T:=[M]$ and where, for every $0\leq i\leq \dim(A), a_i\in \Z$. Let $P$ be a basic set associated to $\sum_{i| a_i> 0} a_i T^i$ and $N$ be the basic set associated to $\sum_{i| a_i\leq 0} -   a_i T^i$. Assume furthermore that both $P$ and $N$ are disjoint from $A$. Then $[N\sqcup A]=[P]$ and both $N$ and $P$ have dimension bounded by $\dim(A)$.
\end{proof}

\begin{lem}\label{suis}
The homomorphism $\pi:(\Z/N\Z)[X]\to K_0(M)$ is surjective.
\end{lem}
\begin{proof}
Lemma \ref{lem_boolsim} and \ref{triv} imply that the class of any definable set in a linear combination over $\Z$ of class of simple sets.
Let $T$ be the class of $M$ in the Grothendieck ring. To conclude we just have to prove that the class of a simple set is an element of $\Z[T]$. This is an immediate consequence of lemma \ref{lem_defbijmd}: a simple set being either a singleton or in definable bijection with a Cartesian power of $M$, its class is of the form $T^k$ with $k\in \N$ (we agree that $0\in\N$ and that $T^0=1$).
\end{proof}

To prove that the Grothendieck ring is $(\Z/N\Z)[X]$ we have to prove that $\pi$ is in fact bijective. Injectivity is the point of the remaining of this article. The following lemma will be used in section 6.

\begin{lem}\label{remplace}
Let $n\in \N$ and let $B$ be a simple subset of $M^{n}$. 
Let $A:=\bigcup_{i=1}^l X_i$ be a finite union of simple sets of $M^{n}$.
Let $F_{1},\ldots, F_{n}$ be finite subsets of $M$ of cardinality a multiple of $N$.
Let us denote $\pi_k: M^n \rightarrow M$ the projection on the $k$-th coordinate.
We assume that for every $1\leq k\leq n$, $F_{k}\cap\pi_k( B)$ and, for every $J\subseteq [|1; l|]$, $ F_{k} \cap\pi_k(\bigcap_{j\in J} X_j\cap B)$ have a cardinality multiple of $N$ (possibly $0$). Then the following holds:
\begin{enumerate}[label=(\alph*)]
\item\label{item4_a} 
$[A\cup \bigcup_{k=1}^{n} B\cap (M^{k-1}\times F_{k}\times M^{n-k})]$ is equal to $[A]$.
\item\label{item4_b} There exist basic sets $P$ and $N$ such that
\begin{itemize}
\item all of whose coefficients are multiples of $N$;
\item their dimension is bounded by $\dim(A\cup B\cap  \bigcup_{k=1}^{n}(M^{k-1}\times F_{k}\times M^{n-k}))$;
\item $ A\cup B\cap \bigcup_{k=1}^{n} (M^{k-1}\times F_{k}\times M^{n-k})\sqcup N$ is in definable bijection with $A\sqcup P$.
\end{itemize}
Furthermore, for any definable set $X$, we can choose $P$ and $N$ to be disjoint from $X$.
\end{enumerate}
\end{lem}
\begin{proof}
\ref{item4_a} By definition, \\
$[A\cup B\bigcup_{k=1}^{n} \cap (M^{k-1}\times F_{k}\times M^{n-k})]=  [A]+ [B\bigcup_{k=1}^{n} \cap (M^{k-1}\times F_{k}\times M^{n-k})]-[A\cap  B\cap \bigcup_{k=1}^{n} (M^{k-1}\times F_{k}\times M^{n-k})].$ 

Because $A$ is a union of simple sets and since the intersection of simple sets is a simple set,
both $ \bigcup_{k=1}^{n} (B\cap (M^{k-1}\times F_{k}\times M^{n-k}))$ and $ \bigcup_{k=1}^{n} (A\cap B\cap (M^{k-1}\times F_{k}\times M^{n+1-k}))$ are of the form $\bigcup_{i\in I} \bigcup_{k=1}^{n} (B_i\cap (M^{k-1}\times F_{i,1}\times M^{n-k}))$ where for every $i\in I$, $B_i$ is a simple set and for every $J\subseteq I$, the cardinality of $F_{k}\cap \pi_k(\bigcap_{j\in J}B_j)$ is a multiple of $N$.

Thus in order to prove the result, it is enough to show the following lemma:
\begin{lem}
Let $F_{1},\ldots, F_{n}$ be finite subsets of $M$.
Let $m\in \N$ and $B_1,\ldots, B_m$ be simple sets of $M^{n}$ such that for every $1\leq k\leq n$, for every $J\subseteq I$, the cardinality of $F_{k}\cap \pi_k(\bigcap_{j\in J}B_j)$ is a multiple of $N$. 
Then $[\bigcup_{i=1}^m \bigcup_{k=1}^{n} (B_i\cap (M^{k-1}\times F_{k}\times M^{n-k}))]$ is equal to $0$. 
\end{lem}

\begin{proof}

By the previous lemma, $[\bigcup_{i=1}^m \bigcup_{k=1}^n (B_i\cap (M^{k-1}\times F_{k}\times M^{n-k}))]$ is a linear combination over $\Z$ of 
the $[\bigcap_{i\in I} \bigcap_{k\in K} (B_i\cap (M^{k-1}\times  \lbrace a_{k}\rbrace \times M^{n-k}))]$ where $I$ ranges over the subsets of $[|1;m|]$, $K$ over the subsets of $[|1;n|]$ and, for every $k\in K$, $a_{k}\in F_{k}$. (We recall that the intersection of simple sets is a simple set and hence, for every  $i\in I$, every $1\leq k\leq n$, every $a_k\in F_k$, 
$B_i\cap (M^{k-1}\times  \lbrace a_{k}\rbrace \times M^{n-k})$ is simple.)

For every $\nu: K \mapsto \cup_{k\in K} F_k$ such that for every $k\in K, \nu(k)\in F_k$, let $c_{I, K, \nu}$ be the coefficient of $[\bigcap_{i\in I} \bigcap_{k\in K} B_i\cap ((M^{k-1}\times  \lbrace\nu(k)\rbrace \times M^{n-k})]$ in this linear combination. 

\medskip
Let us fix $J$ a subset of $[|1;m|]$ and $K$ a subset of $[|1;n|]$. 

Let us consider the set $H_{J,K}:=\lbrace \nu: K \mapsto \cup_{k\in K} F_k | \forall k, \nu(k)\in F_k\cap \pi_k( \bigcap_{j\in J}B_j) \rbrace$. 
The cardinality of this set is equal to the product of the cardinality of the sets $F_{k}\cap \pi_{k}( \bigcap_{j\in J}B_j)$ when $k$ ranges over $K$. Hence its cardinality is a multiple of $N$.

\medskip
Notice that by uniformity, for any  $\nu, \nu'\in H$, the coefficients $c_{J, K, \nu'}$ and $c_{J, K, \nu}$ are equal. As a result, 
$\sum_{\nu\in H_{J,K}} c_{I, K, \nu}[\bigcap_{i\in J} \bigcap_{k\in K} B_i\cap (M^{k-1}\times  \lbrace\nu(k)\rbrace \times M^{n-k})]$ is a multiple of the carddinality of $H_{J,K}$ and hence is a multiple of $N$ in $K_0(M)$. It is thus equal to $0$ in $K_0(M)$.

\medskip
As a result, since 
$$[\bigcup_{i=1}^m \bigcup_{k=1}^{n} (B_i\cap (M^{k-1}\times F_{k}\times M^{n-k}))]=\sum_{J \subset [|1; m|], K\subset [|1;n|]} \sum_{\nu\in H_{J,K}} c_{J, K, \nu}[\bigcap_{i\in J} \bigcap_{k\in K} B_i\cap (M^{k-1}\times  \lbrace\nu(k)\rbrace \times M^{n-k})],$$ it is equal to $0$ in $K_0(M)$.

\ref{item4_b} This can be proved exactly like \ref{item2_b} in lemma \ref{triv}.
\end{proof}

\end{proof}




\section{Normal functions}

\begin{de}\label{de_normal}
We say that a function $g$ defined on a definable subset $A\subseteq M^n$ and ranging into $M^m$ is \textbf{normal} if its graph $\lbrace (x_1,\ldots, x_n, y_1,\ldots, y_m)| g(x_1,\ldots, x_n)=(y_1,\ldots, y_m), (x_1,\ldots, x_n)\in A\rbrace$ can be defined by 
a formula of the form:
$$(x\in A)\wedge (\wedge_{i=1}^m f^{k_i}(y_i)=x_{\sigma(i)})$$
where for every $1\leq i\leq m, k_i\in\Z$ and $\sigma$ is a function from $[\![ 1; m]\!] $ to $[\![ 1; n]\!] $.

We call a formula of this form a \textbf{normal formula}. 
\end{de}

\begin{lem}\label{2}
Let $B\subseteq M^n$ be a simple set, $C\subsetneq B$ a definable subset
and $g$ a normal function defined on $B \smallsetminus C$ by a formula as in lemma \ref{de_normal}. Then $g$ is injective if and only if the image of $\sigma$ contains a representative of every equivalence class of $B$.
\end{lem}

\begin{proof}
Let $I_1,\ldots, I_d$ be the equivalence classes of $B$. 
Let us assume that $\sigma$ does not satisfy the property of the lemma, and let $1\leq t \leq d$ be such that 
no representative of $I_t$ is in the image of $\sigma$.
Then any two elements of $B$, $x$ and $y$ such that for every $i\notin I_t$, $x_i=y_i$ have the same image by $g$. 
The point is then to see that $B \smallsetminus C$ contains two distinct such elements.
If it is not so, then $B\smallsetminus C$ is included in a simple set defined by the conjunction of the formula defining $B$ and a formula fixing each coordinate $x_i$ with $i\in I_t$ to be equal to a constant. This set has thus a dimension strictly less than $B$ which is impossible since $\overline{B\smallsetminus C}$ is equal to $B$ (as $B$ is simple, hence irreducible as a closed set).

Conversely, let us assume that $\sigma$ satisfies this property. If two elements $x, y$ have the same image under $g$ then for every $i\in [\![ 1;n]\!] $, $x_{\sigma(i)}=y_{\sigma(i)}$.
But, since the image of $\sigma$ contains a representative of every equivalence class of $A$, this means that for every $1\leq i\leq n$, $x_i=y_i$.
\end{proof}

\begin{lem}\label{new} 
Let $B$ be a simple set and $C\subsetneq B$ a definable subset.
Let $g$ and $h$ be two normal functions defined on $B$. Let us assume that they coincide on $B\smallsetminus C$. Then they also coincide on $B$.
\end{lem}

\begin{proof}
Let us consider a formula defining $g$:
$$(x\in B\smallsetminus C)\wedge (\wedge_{i=1}^m y_i=f^{k_i}(x_{\sigma(i)}))$$
where for every $1\leq i\leq m, k_i\in\Z$ and $\sigma$ is a function from $[\![ 1; m]\!] $ to $[\![ 1; n]\!] $.

Let $\tilde{g}$ be the function defined by 
$$(x\in B)\wedge (\wedge_{i=1}^m y_i=f^{k_i}(x_{\sigma(i)})).$$

Analogously, $h$ is defined by a formula:
$$(x\in B\smallsetminus C)\wedge (\wedge_{i=1}^m y_i= f^{k'_i}(x_{\sigma'(i)}))$$
where for every $1\leq i\leq m, k'_i\in\Z$ and $\sigma'$ is a function from $[\![ 1; m]\!] $ to $[\![ 1; n]\!] $.

Let $\tilde{h}$ be the function defined by 
$$(x\in B)\wedge (\wedge_{i=1}^m y_i=f^{k'_i}(x_{\sigma'(i)})).$$

Let us check that $\tilde{h}=\tilde{g}$.

Let $1\leq i\leq m$.
For every $x\in B\smallsetminus C$, because $y:=g(x)=h(x)$, we have that
$y_i=f^{k_i}(x_{\sigma(i)})= f^{k'_i}(x_{\sigma'(i)})$. 
To conclude, we just have to notice that for every $x\in B$, every $1\leq i\leq m$,  there exists $x'\in B\setminus C$ such that $x_{\sigma(i)} =x'_{\sigma'(i)}$.
%
\end{proof}

\begin{lem}
 Let $B$ be a simple set and $C\subsetneq B$ a definable subset. Let $g$ be a normal injection $g$ defined on $B\smallsetminus C$ by the normal formula $$(x\in B\smallsetminus C)\wedge (\wedge_{i=1}^m f^{k_i}(y_i)=x_{\sigma(i)}).$$
Then $g$ 
can be uniquely extended on $B$ to a normal injection. This injection is definable by $$(x\in B)\wedge (\wedge_{i=1}^m f^{k_i}(y_i)=x_{\sigma(i)}).$$
\end{lem}

\begin{proof}
The fact that it can be uniquely extended is a direct consequence of lemma \ref{new}.
The fact that the injectivity is preserved is an immediate consequence of \ref{2}.
\end{proof}

\begin{de}
Let $B$ be a simple set and $C\subsetneq B$ a definable subset.
Let $g$ be a normal injection defined on $B\smallsetminus C$.
We call \textbf{normal extension} of $g$ the normal injection that extends $g$ to $B$.
\end{de}


\begin{lem}\label{im}
Let $m, n\in \N$.
Let $B\subseteq M^n$ be a simple set and $g$ a normal injection defined on $B$ ranging into $M^m$. 
Then there exist $B'$ a simple set of dimension $\dim(B)$ and $F_1,\ldots, F_m$ (possibly empty) finite subsets of $M$ of cardinality a multiple of $N$ such that $g(B)=B'\smallsetminus [B'\cap \bigcup_{i=1}^m  (M^{i-1}\times F_i \times M^{m-i})]$. Furthermore, if we denote $\pi_i$ the projection onto the $i$-th coordinate, then for every $1\leq i\leq m$, $F_i\subseteq \pi_i(\overline{g(B)}\smallsetminus g(B))$ and in particular $F_i\subseteq\pi_i(B')$;

\medskip
In addition, the set $B'$ and $F_1,\ldots, F_m$ satisfy that:
\begin{itemize}
\item For any $l\in \N_{>0}$, any $C_1,\ldots, C_l$, simple sets included in $B$,
\[g(B\smallsetminus \bigcup_{j=1}^l C_j)=B'\smallsetminus [B'\cap \bigcup_{i=1}^m (M^{i-1}\times F_i \times M^{m-i})\cup \bigcup_{j=1}^l C'_j]\] where

\begin{itemize}
\item for every $1\leq j\leq l$, $C'_j$ is a simple set included in $B'$;
\item for every $J\subseteq \lbrace 1, \ldots, l\rbrace$, $\bigcup_{i\in J} C_i$ is in definable bijection with $\bigcup_{i\in J} C'_i$, in particular for every $1\leq j\leq l$, $\dim(C'_j)=\dim(C_j)$; \item for every $I\subseteq \lbrace 1,\ldots, m\rbrace$, for every $J\subseteq \lbrace 1,\ldots , l\rbrace$, there exists $N_{I,J}\in \N$ a multiple of $N$ such that $\bigcap_{j\in J} C'_j \cap \bigcap_{i\in I}(M^{i-1}\times F_i \times M^{m-i})$ is the disjoint union of $N_{I,J}$ disjoint simple sets in definable bijection with each other.
\end{itemize}
\end{itemize}
\end{lem}

\begin{proof}
We first prove that $g(B)=B'\smallsetminus [B' \cap (\bigcup_{i=1}^m B'\cap (M^{i-1}\times F_i \times M^{m-i}))]$ where the sets $F_i$ satisfy the properties of the lemma. Let $d:=\dim(B)$, $I_1, \ldots, I_d$ be the equivalence classes of $B$, and for each $1 \leq i \leq d$, $k_i$ a determining element of $I_i$. Let $h$ be the bijection between $M^d$ and $B$ such that $h(x_1,\ldots,x_d)$ is the element $(y_1,\ldots,y_n)$ of $B$ that satisfies for every $1\leq i \leq d,$ $y_{k_i}=x_i$.
The function $g\circ h$ is a bijection from $M^d$ to $g(B)$ and is definable by a formula of the form
$$(y=g\circ h(x)) \equiv [\wedge_{i\in S_1} y_i=f^{k_i}(x_{\alpha(i)})]\wedge [\wedge_{i\in S_2} y_i=c_i]$$ where 
\begin{itemize}
\item $S_1\sqcup S_2$ is a partition of $[|1;m|]$;
\item $\alpha$ is a surjection from $S_1$ onto $[|1;d|]$;
\item for every $i\in S_2$, $c_i$ is a constant.
\end{itemize}

For every $1\leq j\leq d$, let $J_j$ be the set of integers $i\in [1;m]$ such that a formula of the form "$y_i=f^{k}(x_{j})$" is implied by the definition of $h$.
For every $i\in J_j$, we call "associated degree" and denote $k_i$ the integer such that $y_i=f^{k_i}(x_{j})$ (such an integer is unique since $f$ has no cycles). 

For every $1\leq j\leq d$, let $M_j$ be the subset of $J_j$ consisting of elements with minimal
associated degree, say $m_j$.
For every $1\leq j\leq d$, we choose an element $i_{j}$ of $M_j$.

The elements $(y_1,\ldots,y_m)$ of $g(B)$ are then defined by the following conditions:
\begin{itemize}
\item for every $1\leq j\leq d$ and every $i \in I_j$, $y_i=f^{k_i-m_j}(y_{i_{j}})$;
\item for every $1\leq j\leq d$, $y_{i_{j}}\in f^{m_j}(M)$;
\item for every $i\in S_2$, $y_i=c_i$.
\end{itemize}

Let $g'$ be the injection ranging into $M^m$ defined on $B$ by $$(y=g'(x)) \equiv [x\in B]\wedge [\wedge_{j=1}^d \wedge_{i\in J_j} y_i=f^{k_i-m_j}(x_{j})]\wedge [\wedge_{i\in S_2} y_i=c_i].$$ 

Then the image by $g'$ of $B$ is the simple set $B'$ defined by $(y_1,\ldots, y_m)\in B'$ if and only if:
\begin{itemize}
\item for every $1\leq j\leq d$, and every $i \in I_j$, $y_i=f^{k_i-m_j}(y_{i_{j}})$;
\item for every $i\in S_2$, $y_i=c_i$.
\end{itemize}

Let us notice that for every $1\leq j\leq d$, every $i \in J_j$ there exists a finite set
of cardinality a multiple of $N$, that we denote $F_{i}$, such that $f^{k_i}(M)=M\smallsetminus F_i$.
For every $i\in S_2$, we put $F_i:=\emptyset$.
Then
$g(B)=B'\smallsetminus \bigcup_{i=1}^m  \tilde{B}_i$ where for every $i\in [|1;m|]$, $ \tilde{B}_i:=B'\cap (M^{i-1}\times F_{i}\times M^{m-i})$.
It is immediate to check that $F_i\subseteq\pi_i(\overline{g(B)}\smallsetminus g(B))$.

\medskip
\noindent
We now prove that for any $l\in \N_{>0}$ any simple sets $C_1,\ldots, C_l$ included in $B$.
Since, for every $1\leq i\leq m$, the sets $F_{i}$ are only dependent on the integers $k_i$ that appear in the definition of $g$, for any simple set $C\subseteq B$, $g(C)=C'\smallsetminus \bigcup_{i=1}^m  \tilde{C}_i$ where 
\begin{itemize}
\item $C'$ is a simple set included in $B'$ whose dimension is equal to that of $C$;
\item for every $i \in [|1;m|]$, $\tilde{C}_i:=C'\cap (M^{i-1}\times F_{i}\times M^{m-i})\subseteq B'\cap (M^{i-1}\times F_i\times M^{m-i})$.
\end{itemize}

For every $1\leq j\leq l$, let $C'_j=\overline{g(C_j)}$.
Then for every $J\subseteq \lbrace 1, \ldots, l\rbrace$, the set $\bigcup_{i\in J} C_i$ is in definable bijection with $g'(\bigcup_{i\in J} C_i)=\bigcup_{i\in J} C'_i$.
Hence, if $l\in \N_{>0}$ and $C_1,\ldots, C_l$ are simple sets included in $B$, then 
$$g(B\smallsetminus \bigcup_{j=1}^l C_j)=B'\smallsetminus (\bigcup_{i=1}^m (M^{i-1}\times F_{k_i}\times M^{m-i})\cup \bigcup_{j=1}^l C'_j)$$
where 
\begin{itemize}
\item for every $1\leq j\leq l$, $C'_j$ is a simple set included in $B'$;
\item for every $J\subseteq \lbrace 1, \ldots, l\rbrace$, $\bigcup_{i\in J} C_i$ is in definable bijection with $\bigcup_{i\in J} C'_i$, in particular for every $1\leq j\leq l$, $\dim(C'_j)=\dim(C_j)$ where $C'_j=\overline{g(C_j)}$.
\end{itemize}

\medskip
Let us now prove the last point.
Let us notice that for any simple set $C$, for every $1\leq i\leq m$, 
$F_i \subseteq \pi_i(\overline{g(C)})$.

In particular, for every $1\leq i\leq m$, for every $J\subseteq \lbrace 1,\ldots ,m\rbrace$, $F_i\subseteq \pi_i(\overline{g(\bigcap_{j\in J} C_j)})=\pi_i(\bigcap_{j\in J} C'_j)$. 

For simplicity, let us denote $C:=\bigcap_{j\in J} C'_j$ and $|F_i|$ the cardinalit of $F_i$.
Since for every $1\leq i\leq m$, $F_i\subseteq\pi_i(C)$, 
$\bigcap_{i\in I} C'\cap(M^{i-1}\times F_i\times M^{m-i})$ is the disjoint union of $\Pi_{i\in I| F_i\neq \emptyset} |F_i|$ simple sets in definable bijection with each other. More precisely, 
$\bigcap_{i\in I} C'\cap(M^{i-1}\times F_i\times M^{m-i})= \bigsqcup_{(x_i)_{i\in I}\in \Pi_{i\in I} F_i} C'\cap(M^{i-1}\times \lbrace x_i\rbrace \times M^{m-i})$.



\end{proof}

\section{Decomposition of definable sets adapted to a definable injection}

\begin{lem}
Let $A$ be a definable set and $g$ a definable injection defined on $A$.
Then there exists a finite Boolean combination of simple sets, $\bigsqcup_{i\in I} (B_i\smallsetminus \bigcup_{j\in I_i} B_{i,j})$ such that:
\begin{itemize}
\item for every $i, j$, $B_{i,j} \subseteq B_i$ are simple sets;
\item $\bigsqcup_{i\in I} (B_i\smallsetminus \bigcup_{j\in I_i} B_{i,j})$ is a disjoint union of $A$;
\item $g$ restricted to $B_i\smallsetminus \bigcup_{j\in I_i}  B_{i,j}$ is a normal function.
\end{itemize}
\end{lem}

\begin{proof}
By quantifier elimination it is clear that there exists a definable partition of $A$, $A=\bigsqcup_{i\in I'} A_i$ such that for every $i$, $g$ restricted to $A_i$ is a normal function.
By lemma \ref{lem_boolsim}, for every $i\in I'$, there exists a finite Boolean combination of simple sets, $\bigsqcup_{l\in J_i} (B_l\smallsetminus \bigcup_{j\in J_{i,j}} B_{l,j})$ such that:
\begin{itemize} 
\item $\bigsqcup_{l\in J_i}(B_l\smallsetminus \bigcup_{j\in  J_{i,j}} B_{l,j})$ is a partition of $A'_i$;
\item for every $l, j$, $B_{l,j} \subseteq B_l$ are simple sets.
\end{itemize}

As $A=\bigsqcup_{i\in I'} A_i$ is a disjoint union, the results follows.
\end{proof}

\begin{de}
Let $A$ be a definable set and $g$ a definable injection defined on $A$.
A partition $\bigsqcup_i(B_i\smallsetminus \bigcup_j B_{i,j})$ of $A$ that satisfies the conditions of the former lemma is called \textbf{decomposition of $A$ adapted to $g$}.
For every $i,j$, the subsets $B_i$ are called \textbf{positive} and the subsets $B_{i,j}$ are called \textbf{negative} (relatively to this partition).
For every $i$, $g_i$ denotes the normal extension of $g$ to $B_i$.
\end{de}

\begin{lem}\label{beau}
Let $A$ be a definable subset of $M^n$ and $g$ a definable injection defined on $A$ ranging in a definable subset of $M^m$ with $m, n\in \N$.
Let us assume that $A=\bigsqcup_{i\in I} (B_i\smallsetminus \bigcup_{j\in I_i} B_{i,j})$ is a partition of $A$ adapted to $g$. 
Then for every $i\in I$, and for every $j\in I_i$, there exist $B'_i$ and $B'_{i,j}\subsetneq B'_i$ such that:

\begin{enumerate}[label=(\alph*)]
\item\label{item2_a} 
\begin{itemize}
\item $B'_i$ is a simple set of dimension equal to that of $B_i$;
\item for every $J\subseteq  I_i$, the set $\bigcup_{i\in J} B_{i,j}$ is in definable bijection with $\bigcup_{i\in J} B'_{i,j}$;
\item there exist $F_{i,1},\ldots, F_{i,m}$ (possibly empty)
finite subsets of $M$ of cardinality a multiple of $N$ such that 
\begin{itemize}
\item for every $i$, $$g(B_i\smallsetminus \bigcup_{j\in I_i} B_{i,j})=B'_i\smallsetminus (\bigcup_{j\in I_i} B'_{i,j}\bigcup_{k=1}^m (B'_i\cap (M^{k-1}\times F_{i,k} \times M^{m-k})));$$
\item for every $i\in I$, every $j\in I_i$, every $k\in \lbrace 1,\ldots, m\rbrace$,
$F_i\subseteq \pi_i(B_i\cap B_{i,j})$ where $\pi_i: M^m \rightarrow M$ is the projection on the $i$-th coordinate.

\end{itemize}
\end{itemize}

\item\label{item2_b} For every $i\in I$, there exist four basic sets, $E, E', C, C'$ such that
\begin{itemize} 
\item $E, E', C, C'$ are disjoint from each other;
\item their associated polynomials are of degree at most $\dim(B_i)-1$;
\item the associated polynomials of $E'$ and $C'$ have coefficients multiple of $N$;
\item $E\sqcup \bigcup_{j\in I_i} B_{i,j}$ is in definable bijection with $C$ and
\[E\sqcup E' \sqcup \bigcup_{j\in I_i} B'_{i,j}\cup \bigcup_{k=1}^m (B'_i\cap (M^{k-1}\times F_{i,k} \times M^{m-k}))\] is in definable bijection with $C\sqcup C'$.
\end{itemize} 
Furthermore, for every definable set $X$, we can choose the sets $E, E', C, C'$ to be disjoint from $X$.
\end{enumerate}
\end{lem}

\begin{proof}
The point \ref{item2_a} is a direct application of Lemma \ref{im}.

Let us prove \ref{item2_b}. Let $i\in I$.
The existence of basic sets $E$ and $C$ such that
\begin{itemize}
\item their dimension is at most $\dim(B_i)-1$;
\item $E\sqcup \bigcup_{j\in I_i} B_{i,j}$ is in definable bijection with $C$
\end{itemize}
is a direct application of lemma \ref{triv}. We can furthermore choose $E$ disjoint from $\bigcup_{j\in I_i} B'_{i,j}$.
Since $\bigcup_{j\in I_i} B_{i,j}$ is in definable bijection with $\bigcup_{j\in I_i} B'_{i,j}$, we also have that $E\sqcup \bigcup_{j\in I_i} B'_{i,j}$ is in definable bijection with $C$.

Since for every $j,k$, $\pi_k(F_k\cap B'_{i,j}\cap B'_i)$ has cardinality a multiple of $N$, we can apply lemma \ref{remplace}.
There exist $E', C$ basic sets such that
\begin{itemize}
\item they are disjoint from $C\sqcup \bigcup_{j\in I_i} B'_{i,j}\cup \bigcup_{k=1}^m (B'_i\cap (M^{k-1}\times F_{i,k} \times M^{m-k}))$; 
\item their associated polynomial have all their coefficients multiple of $N$;
\item $E' \sqcup \bigcup_{j\in I_i} B'_{i,j}\cup \bigcup_{k=1}^m (B'_i\cap (M^{k-1}\times F_{i,k} \times M^{m-k}))$ is in definable bijection with $\bigcup_{j\in I_i} B'_{i,j} \sqcup C'$.
\end{itemize} 
Hence,
$E\sqcup E'\sqcup \bigcup_{j\in I_i} B'_{i,j}\cup \bigcup_{k=1}^m (B'_i\cap (M^{k-1}\times F_{i,k} \times M^{m-k}))$ is in definable bijection with $C \sqcup C' $.
\end{proof}

\section{Computation of the Grothendieck ring}

Let $T=[M]$ be the class of $ M$ in its Grothendieck ring.
We now conclude the proof of Theorem \ref{but2} proving that the homomorphism $\pi:(\Z/N\Z)[X]\to K_0(M)$ defined in Section \ref{section_simple}, mapping $X$ to $T$, is bijective. 

By Lemma \ref{suis}, $\pi$ is surjective.
To prove injectivity, it is sufficient to show that
for every distinct polynomials $ \sum_{k = 0}^d \bar{a_k} X^k $ and $ \sum_{k = 0}^d \bar{b_k} X^k $ of $ (\Z / N \Z) [X] $, we have
$ \sum_{k = 0}^d \bar{a_k} T^k \neq \sum_{k = 0}^{d'} \bar{b_k} T^k $ in $ K_0(M) $.

First let us show that this is equivalent to proving that if $A$ and $B$ are two basic sets in definable bijection with each other, then their associated polynomials have the same reduct in $(\Z/N\Z)[X]$.

Let $P(X), Q(X)\in \Z[X]$. Let $\overline{P}(X)$ (respectively $\overline{Q}(X)$) be the reduct of $P(X)$ (respectively $Q(X)$) in $(\Z/N\Z)[X]$. Let us assume that in $K_0(M)$, $\overline{P}(T)=\overline{Q}(T)$.
Let $A$ (respectively $B$) be a basic set associated to $P(X)$ (respectively $Q(X)$).

Since $\overline{P}(T)=\overline{Q}(T)$, $A$ and $B$ have the same class in $K_0(M)$ and there exists a definable set $C$ such that $A\sqcup C$ and $B\sqcup C$ are in definable bijection.

By lemma \ref{triv}, there exists $D$, disjoint from $A, B$ and $C$, and such that $C\sqcup D$ is in definable bijection with a basic set, say $D'$, disjoint from $A$ and $B$.

Hence, $A\sqcup C\sqcup D$  is in definable bijection with $A\sqcup D'$, and  $B\sqcup C\sqcup D$  is in definable bijection with $B\sqcup D'$.  As a result $A\sqcup D'$ and $B\sqcup D'$ are two basic sets in definable bijection with each other.
It is immediate to notice that the union of two disjoint basic sets is in definable bijection with a basic set whose associated polynomial is the sum of their own associated polynoms.

Let us assume that we have proved the following fact: two basic sets are in definable bijection if and only if their associated polynomials are equal modulo $N$.  Let $R(X)$ be the associated polynomial of $D'$. The fact that  $A\sqcup D'$ and $B\sqcup D'$ are two basic sets in definable bijection with each other, imply that $P(X)+R(X)$ and $Q(X)+R(X)$ are equal modulo $N$, which implies that $P(X)$ is equal to $Q(X)$ modulo $N$.

To end the demonstration, we thus only need to prove the following lemma.
\begin{lem}
Let $(d_n)_{n\in\N}$ be a sequence of distinct elements of $M$.
Let $d\in \N$ and for every $1\leq k\leq d$, $a_k\in \Z$.
Let $d'\in \N$ and for every $1\leq k\leq d'$, $b_k\in \Z$.
Let us assume that
$$U:=\bigsqcup_{k=0}^d \bigsqcup_{i = 1}^{a_k} \lbrace d_i \rbrace \times M^k \times \lbrace (d_0, \ldots, d_{n-k}) \rbrace$$
is in definable bijection with
$$V:=\bigsqcup_{k=0}^{d'} \bigsqcup_{i = 1}^{b_k} \lbrace d_i \rbrace \times M^k \times \lbrace (d_0, \ldots, d_{n-k}) \rbrace.$$
Then, $d=d'$ and, for every $1\leq k\leq d$, $a_k\equiv b_k$ modulo $N$. 
\end{lem}

\begin{proof}

Let us first notice that, by an immediate consequence of \ref{beau}, two sets that are in definable bijection have the same dimension. As a result, $d=d'$.

Let us prove by induction on $d$ that for every $k$, $a_k\equiv b_k$ modulo $N$.
If $d=0$, then the result is clear since $U$ is finite.

Let $d$ be an integer greater than $0$. Let us suppose the result proved for any integer smaller than $d$.
Let $g$ be a definable bijection from $U$ to $V$.
By Lemma \ref{beau}, $U$ admits a decomposition adapted to $g$,
$U=\bigsqcup_i (B_i\smallsetminus \bigcup_{j\in J_i} B_{i,j})$, such that for every $i\in I$, there exists $B'_i$ and for every $j\in I_i$ there exists $B'_{i,j}\subseteq B'_i$ that satisfy:
\begin{itemize}
\item $B'_i$ is a simple set of dimension equal to that of $B_i$;
\item for every $J\subseteq  I_i$, the set $\bigcup_{i\in J} B_{i,j}$ is in definable bijection with $\bigcup_{i\in J} B'_{i,j}$;
\item there exist $F_{i,1},\ldots, F_{i,m}$
(possibly empty) finite sets of cardinality a multiple of $N$ such that 
\begin{itemize}
\item for every $i$, $$g(B_i\smallsetminus \bigcup_{j\in I_i} B_{i,j})=B'_i\smallsetminus (\bigcup_{j\in I_i} B'_{i,j}\cup \bigcup_{k=1}^m (B'_i\cap (M^{k-1}\times F_{i,k} \times M^{m-k})));$$
\item for every $i\in I$, every $j\in I_i$, every $k\in \lbrace 1,\ldots, m\rbrace$,
$F_i\subseteq \pi_i(B_i\cap B_{i,j})$ where $\pi_i$ is the projection on the $i$-th coordinate.

\end{itemize}
\end{itemize}

The positive subsets of the decomposition $\bigsqcup_i (B_i\smallsetminus \bigcup_{j\in J_i} B_{i,j})$ that are of maximal dimension correspond to the sets $\lbrace d_i \rbrace \times M^d \times \lbrace d_0 \rbrace$ for $1\leq i\leq a_d$.
Let $B_1,\ldots, B_r$ be these subsets.

For every $i$, $B'_i= g_i(B_i)$ is a simple set of dimension equal to $\dim(B_i)$. The irreducible components of $B$ of dimension $d$ are thus the subsets $B'_1,\ldots, B'_r$ and we have the equality $a_d=b_d$.
The function $g$ is a bijection from $\bigsqcup_{\lbrace i|\dim(B_i)<d\rbrace} (B_i\smallsetminus \bigcup_{j\in J_i} B_{i,j})$ onto $\bigsqcup_{\lbrace i|\dim(B_i)<d\rbrace} (B'_i\smallsetminus \bigcup_j B'_{i,j}).$

Let us decompose the set 
$$\bigsqcup_{\lbrace i| \dim(B_i)<d\rbrace} (B_i\smallsetminus \bigcup_{j\in J_i} B_{i,j}).$$ It is the disjoint union of 
$$X_1:=\bigsqcup_{\lbrace i| \dim(B_i)<d \,\wedge\, B_i \not \subseteq \bigcup_{l=1}^r B_l\rbrace} (B_i\smallsetminus \bigcup_{j\in J_i} B_{i,j})$$ and 
$$X_2:=\bigsqcup_{\lbrace i| \dim(B_i)<d\,\wedge\, B_i \subseteq \bigcup_{l=1}^r B_l\rbrace} (B_i\smallsetminus \bigcup_{j\in J_i} B_{i,j}).$$

Similarly, the set

 $$\bigsqcup_{\lbrace i |\dim(B_i)<d\rbrace} \bigcup_{j\in J_i} B'_{i,j}
\cup
\bigcup_{k=1}^m (B'_i\cap (M^{k-1}\times F_{i,k} \times M^{m-k}))
$$ can be decompose as the disjoint union of

$$Y_1:=\bigsqcup_{\lbrace i|\dim(B'_i)<d \,\wedge\, B'_i \not \subseteq \bigcup_{l=1}^r B'_l\rbrace}
 B'_i\smallsetminus  (
 \bigcup_{j\in J_i} B'_{i,j}
\cup
\bigcup_{k=1}^m (B'_i\cap (M^{k-1}\times F_{i,k} \times M^{m-k})))
$$ and 
~
$$Y_2:=\bigsqcup_{\lbrace i|\dim(B'_i)<d\,\wedge\, B'_i \subseteq \bigcup_{l=1}^r B'_l\rbrace} 
 B'_i\smallsetminus 
(\bigcup_{j\in J_i} B'_{i,j}
\cup
\bigcup_{k=1}^m (B'_i\cap (M^{k-1}\times F_{i,k} \times M^{m-k}))).$$

Let us notice that
\[X_1=\bigsqcup_{k=0}^{d-1} \bigsqcup_{i = 1}^{a_k} \lbrace d_i \rbrace \times M^k \times \lbrace (d_0, \ldots, d_{n-k}) \rbrace;\quad\]
and
\[Y_1=\bigsqcup_{k=0}^{d-1} \bigsqcup_{i = 1}^{b_k} \lbrace d_i \rbrace \times M^k \times \lbrace (d_0, \ldots, d_{n-k}) \rbrace.\]

Furthermore, since $U$ is closed, $X_2$ is equal to the complement of $\bigsqcup_{i=1}^r B_i\smallsetminus \bigcup_{J\in J_i} B_{i,j}$ in $U$.

Because the sets $B_1,\ldots, B_r$ are disjoint, 
\begin{align*}
X_2=\bigsqcup_{i=1}^r \bigcup_{j\in J_i} B_{i,j}.
\end{align*}

Similarly, since $V$ is closed, 

$$
Y_2=\bigsqcup_{i=1}^r 
\bigcup_{j\in J_i} B'_{i,j}
\cup \bigcup_{k=1}^m (B'_i\cap (M^{k-1}\times F_{i,k} \times M^{m-k})).$$



According to Lemma \ref{beau}, for every $i\in I$, there exist four basic sets, $E_i, E'_i, C_i, C'_i$ such that
\begin{itemize}
\item  their associated polynomial have degree at most $\dim(B_i)-1$;
\item they are disjoint from $U\cup \overline{g(U)}$ and disjoint from each other;
\item the set $E_i\sqcup \bigcup_{j\in J_i} B_{i,j}$ is in definable bijection with $C_i$;
\item the coefficients of the polynomial associated to $E'_i$ (respectiveley $C'_i$) are multiple of $N$;
\item the set
\[E_i\sqcup E'_i \sqcup \bigcup_{j\in J_i} B'_{i,j}\cup \bigcup_{k=1}^m (B'_i\cap (M^{k-1}\times F_{i,k} \times M^{m-k}))\] is in definable bijection with $C_i\sqcup C'_i$;
\item for every distinct elements $i, j$ of $\lbrace 1,\ldots, r\rbrace$, for every $Z_i\in \lbrace E_i, E'_i, C_i, D_i \rbrace$, every $Z_j\in \lbrace E_j, E'_j, C_j, D_j \rbrace$, the set $Z_i\cap Z_j$ is empty.
\end{itemize}

Since for every distinct elements $i, j$ of $\lbrace 1,\ldots, r\rbrace$, $B_i\cap B_j =\emptyset$, the set $\bigsqcup_{i=1}^r E_i\sqcup X_2$ is in definable bijection with
$\bigsqcup_{i=1}^r C_i$
and $\bigsqcup_{i=1}^r E_i \sqcup \bigsqcup_{i=1}^r E'_i \sqcup Y_2$ is in definable bijection with the disjoint union 
$\bigsqcup_{i=1}^r C_i \sqcup \bigsqcup_{i=1}^r C'_i.$

As a result, $X_1\sqcup X_2\sqcup \bigsqcup_{i=1}^r E_i\sqcup \bigsqcup_{i=1}^r E'_i$ is in definable bijection with 
\begin{align*}
X_1\sqcup \bigsqcup_{i=1}^r C_i\sqcup \bigsqcup_{i=1}^r E'_i.
\end{align*}

Similarly, $Y_1\sqcup Y_2\sqcup\bigcup_{i=1}^r E_i\sqcup \bigsqcup_{i=1}^r E'_i$ is in definable bijection with
\begin{align*}
Y_1\sqcup \bigsqcup_{i=1}^r C_i\sqcup \bigsqcup_{i=1}^r C'_i.
\end{align*}

Let
\begin{itemize}
\item $P(X)$ be the polynomial associated to $X_1$,
\item $Q(X)$ be the polynomial associated to $Y_1$,
\item $R(X)$ be the polynomial associated to $\bigsqcup_{i=1}^r C_i$,
\item $S(X)$ be the polynomial associated to $\bigsqcup_{i=1}^r C'_i$,
\item $T(X)$ be the polynomial associated to $\bigsqcup_{i=1}^r E'_i$.
\end{itemize}

As the union of disjoint basic sets is a basic set whose associated polynomial is the sum of the associated polynomials of each term of this union, the associated polynomial of $X_1\sqcup \bigsqcup_{i=1}^r C_i\sqcup \bigsqcup_{i=1}^r E'_i$, and hence of $X_1\sqcup X_2\sqcup \bigsqcup_{i=1}^r E_i\sqcup \bigsqcup_{i=1}^r E'_i$ is $P(X)+R(X)+T(X)$.

Similarly the one of $Y_1\sqcup Y_2\bigsqcup\bigcup_{i=1}^r E_i\sqcup \bigsqcup_{i=1}^r E'_i$ is $Q(X)+R(X)+S(X)$.

Since, $X_1\sqcup X_2\sqcup \bigsqcup_{i=1}^r E_i\sqcup \bigsqcup_{i=1}^r E'_i$ and $Y_1\sqcup Y_2\bigsqcup\bigcup_{i=1}^r E_i\sqcup \bigsqcup_{i=1}^r E'_i$ are basic sets of dimension at most $d$ that are in definable bijection with each other, we can apply the  induction hypothesis and get 
\begin{align*}
P(X)+R(X)+T(X)= Q(X)+R(X)+S(X) \text{ modulo }N.
\end{align*}

Since the coefficients of $S(X)$ and $T(X)$ are multiple of $N$, this leads to
\begin{align*}
P(X)=Q(X) \text{ modulo }N.
\end{align*}
We already know that $a_d=b_d$, this means that the polynomial associated to $U$ is equal to the one associated to $V$ modulo $N$.

\end{proof}

\medskip
\noindent {\bf Acknowledgements.} 
I am very grateful to Fran\c{c}oise DELON who has carefully read several previous drafts and has devoted a huge amount of time listing the many typos of the first versions. I also want to thank Yves DE CORNULIER and David MARKER for having read my preprints and having made useful suggestions to improve them.

\end{document}